\documentclass[a4paper,11pt]{article}

\usepackage[fleqn]{amsmath}
\usepackage{amsthm, amssymb, amsfonts, mathtools, stmaryrd, mathrsfs}
\usepackage[T1]{fontenc}
\usepackage{enumitem}
\usepackage[english]{babel}
\usepackage{etoolbox}
\usepackage{thmtools}
\usepackage{yfonts}
\usepackage{nameref}
\usepackage{bbm}
\usepackage{cleveref}

\usepackage{fullpage}
\setlist{nosep}
\usepackage{titlesec}
\titleformat{\section}{\large\centering\scshape}{\thesection.\quad}{0em}{}
\titlespacing*{\section}{0mm}{1mm}{1mm}

\declaretheoremstyle[
    headfont=\bfseries, 
    notebraces={--- \makefirstuc}{},
    notefont=\normalfont\itshape,
    bodyfont=\normalfont,
    headpunct={\vspace{0.15\topsep}},
    spaceabove=\topsep,
    spacebelow=0pt,
    postheadspace=\newline,
    qed=$\triangleleft$,
]{mystyle}
\theoremstyle{mystyle}

\declaretheorem[style=mystyle,name=Theorem]{thm}
\crefname{thm}{theorem}{theorems}
\Crefname{thm}{Theorem}{Theorems}
\declaretheorem[style=mystyle,name=Definition,sibling=thm]{dfn}
\crefname{dfn}{definition}{definitions}
\Crefname{dfn}{Definition}{Definitions}
\declaretheorem[style=mystyle,name=Lemma,sibling=thm]{lmm}
\crefname{lmm}{lemma}{lemmas}
\Crefname{lmm}{Lemma}{Lemmas}
\declaretheorem[style=mystyle,name=Corollary,sibling=thm]{crl}
\crefname{crl}{corollary}{corollaries}
\Crefname{crl}{Corollary}{Corrolaries}

\crefname{prp}{proposition}{propositions}
\Crefname{prp}{Proposition}{Propositions}

\crefname{exm}{example}{examples}
\Crefname{exm}{Example}{Examples}

\crefname{rmk}{remark}{remarks}
\Crefname{rmk}{Remark}{Remarks}

\crefname{rmd}{reminder}{reminders}
\Crefname{rmd}{Reminder}{Reminders}

\crefname{pdx}{paradox}{paradoxes}
\Crefname{pdx}{Paradox}{Paradoxes}

\crefname{clm}{claim}{claims}
\Crefname{clm}{Claim}{Claims}
\declaretheorem[style=mystyle,name=Fact,sibling=thm]{fct}
\crefname{fct}{fact}{facts}
\Crefname{fct}{Fact}{Facts}

\crefname{qst}{question}{questions}
\Crefname{qst}{Question}{Questions}
\declaretheorem[style=mystyle,name=Theorem,numbered=no]{thm*}
\declaretheorem[style=mystyle,name=Definition,numbered=no]{dfn*}
\declaretheorem[style=mystyle,name=Lemma,numbered=no]{lmm*}
\declaretheorem[style=mystyle,name=Corollary,numbered=no]{crl*}
\declaretheorem[style=mystyle,name=Proposition,numbered=no]{prp*}
\declaretheorem[style=mystyle,name=Example,numbered=no]{exm*}
\declaretheorem[style=mystyle,name=Remark,numbered=no]{rmk*}
\declaretheorem[style=mystyle,name=Reminder,numbered=no]{rmd*}
\declaretheorem[style=mystyle,name=Paradox,numbered=no]{pdx*}
\declaretheorem[style=mystyle,name=Question,numbered=no]{qst*}
\declaretheorem[style=mystyle,name=Claim,numbered=no]{clm*}

\makeatletter
\renewenvironment{proof}[1][\proofname]{
  \pushQED{\qed}%
  \normalfont%
  \topsep0pt \partopsep0pt%
  \trivlist%
  \item[\hskip\labelsep%
        \itshape%
    #1\@addpunct{.}]\ignorespaces%
}{%
  \popQED\endtrivlist\@endpefalse%
  \addvspace{0pt}%
}
\makeatother

\newenvironment{subproof}{\begin{proof}[Proof of claim]}{\renewcommand{\qedsymbol}{$\blacksquare$}\end{proof}}

\let\phi\varphi
\let\epsilon\varepsilon
\let\emptyset\varnothing

\let\subset\subseteq
\let\supset\supseteq

\let\bar\overline

\renewcommand{\c}{\mathcal}
\newcommand{\bb}{\mathbb} 
 
\renewcommand{\b}{\mathbf}
\newcommand{\s}{\mathsf}

\renewcommand{\r}{\mathrm}
\newcommand{\f}{\mathfrak}

\newcommand{\lo}{\lor}				
\newcommand{\emp}{\emptyset}		
\renewcommand{\Cap}{\bigcap}		
\renewcommand{\Cup}{\bigcup}		
\newcommand{\La}{\bigwedge\!} 		
	
\newcommand{\ent}{\vdash} 					
\newcommand{\md}{\vDash} 					
\newcommand{\fc}{\Vdash} 					

\newcommand{\ab}[1]{\left\langle#1\right\rangle} 		
\newcommand{\ap}[1]{\text{``}\,#1\,\text{''}} 			
\newcommand{\card}[1]{\left|#1\right|} 					
\newcommand{\st}[1]{\left\{#1\right\}} 					
\newcommand{\smid}{\ \middle |\ } 						

\newcommand{\ran}{\r{ran}}
\newcommand{\ot}{\r{ot}}

\newcommand{\cof}{\r{cof}}
\newcommand{\add}{\r{add}}

\newcommand{\id}{\s{id}}
\newcommand{\pow}{\s{pow}}
\newcommand{\Loc}{\r{Loc}}
\newcommand{\Split}{\r{Split}}

\newcommand{\suc}{\r{suc}}
\newcommand{\supp}{\r{supp}}

\newcommand{\gcs}{{}^{\kappa}2}

\newcommand{\bs}{{}^{\omega}\omega}

\newcommand{\gbs}{{}^{\kappa}\kappa}
\newcommand{\gfbs}{{}^{<\kappa}\kappa}

\newcommand{\forcing}[1]{{\bb S^{#1}_\kappa}}

\newcommand{\dstar}[1]{\f d_\kappa^{#1}(\in^*)}

\newcommand{\ft}{\mathbbm 1}

\renewcommand{\emph}{\textbf}
\title{\sc Separating Many Localisation Cardinals\\on the Generalised Baire Space}
\author{Tristan van der Vlugt\thanks{{\sc Universität Hamburg, Fachbereich Mathematik} $-$ tristan.van.der.vlugt@uni-hamburg.de $-$ The author wishes to thank Jörg Brendle for valuable discussions, suggestions and extensive comments.}}
\date{\vspace{-5mm}March 5, 2022}

\begin{document}

\maketitle\vspace{-10mm}

\begin{abstract}
\noindent Given a cofinal cardinal function $h\in\gbs$ for $\kappa$ inaccessible, we consider the dominating $h$-localisation number, that is, the least cardinality of a dominating set of $h$-slaloms such that every $\kappa$-real is localised by a slalom in the dominating set. It was proved in \cite{BBFM} that the dominating localisation numbers can be consistently different for two functions $h$ (the identity function and the power function). We will construct a $\kappa$-sized family of functions $h$ and their corresponding localisation numbers, and use a ${\leq}\kappa$-supported product of a cofinality-preserving forcing to prove that any simultaneous assignment of these localisation numbers to cardinals above $\kappa$ is consistent. This answers an open question from \cite{BBFM}.
\end{abstract}

In an effort to generalise the cardinal characteristics related to the null ideal from the context of the continuum $\bs$ to the generalised Baire space $\gbs$, the authors of \cite{BBFM} considered localisation cardinals. These cardinals were first described in the context of $\bs$ by Tomek Bartoszy\'nski \cite{Bart87}, and are defined using slaloms. Let $\kappa$ be a regular strong limit cardinal (hence $\kappa$ is inaccessible or equal to $\omega$) and let $h\in\gbs$. An \emph{$h$-slalom} is any function $\phi:\kappa\to [\kappa]^{<\kappa}$ such that $|\phi(\alpha)|\leq |h(\alpha)|$ for all $\alpha\in\kappa$.

For $f\in\gbs$, we say $f\in^*\phi$, or $f$ is \emph{localised} by $\phi$, if there exists some $\xi<\kappa$ such that $f(\alpha)\in \phi(\alpha)$ for all $\alpha\in[\xi,\kappa)$. We denote the set of all $h$-slaloms by $\Loc_h$.

Using these concepts, we can define the following two cardinal characteristics, sometimes called \emph{localisation cardinals}:
\begin{align*}
\f b_\kappa^h(\in^*) &= \text{the least cardinality of a family }\c F\subset \gbs\text{ such that }\forall \phi\in \Loc_h \exists f\in\c F(f\notin^*\phi),\\
\f d_\kappa^h(\in^*) &= \text{the least cardinality of a family }\Phi\subset \Loc_h\text{ such that }\forall f\in \gbs \exists \phi\in\Phi(f\in^*\phi).
\end{align*}
In the case that $\kappa=\omega$ these cardinals give a combinatorial definition of two of the cardinal invariants of the Lebesgue null ideal $\c N$:
\begin{align*}
\add(\c N) &= \text{the least cardinality of a family }\c A\subset \c N\text{ such that }\Cup \c A\notin \c N,\\
\cof(\c N) &= \text{the least cardinality of a family }\c C\subset \c N\text{ such that }\forall N\in\c N\exists C\in \c C(N\subset C)
\end{align*}
Bartoszy\'nski introduced slaloms in \cite{Bart87} to give the following combinatorial definition to $\add(\c N)$ and $\cof(\c N)$.

\begin{fct}
\renewcommand{\qedsymbol}{$\square$}
$\add(\c N)=\f b_\omega^h(\in^*)$ and $\cof(\c N)=\f d_\omega^h(\in^*)$.\qedhere
\end{fct}

The choice of $h\in\bs$ is irrelevant here, as it does not influence the cardinality of $\f b_\omega^h(\in^*)$ and $\f d_\omega^h(\in^*)$. This turns out to be different in the generalised case. Recently,  it was proved in \cite{BBFM} that $\dstar h$ can consistently have different values for different $h\in\gbs$. In particular, it was shown that $\dstar{\pow}<\dstar{\id}$ is consistent, where $\pow:\alpha\mapsto 2^\alpha$ and $\id:\alpha\mapsto \alpha$. 

This consistency was proved using a generalisation of Sacks forcing on $\gcs$ that has the generalised $\pow$-Sacks property, but not the generalised $\id$-Sacks property.

\begin{dfn}
Let $h\in\gbs$. A forcing $\bb P$ has the \emph{generalised $h$-Sacks property} if for every $\bb P$-name $\dot f$ and $p\in \bb P$ such that $p\fc\ap{\dot f:\check\kappa\to\check\kappa}$ there exists an $h$-slalom $\phi$ and $q\leq p$ such that $q\fc\ap{\dot f(\check \alpha)\in\check \phi(\check\alpha)}$ for all $\alpha<\kappa$. We will from now on simply say $h$-Sacks property and omit ``generalised''.
\end{dfn}

Hence, if $\Phi\subset\Loc_h$ is a family as in the definition of $\dstar h$ in the ground model, then $\Phi$ will still witness the size of $\dstar h$ in the extension by a forcing with the $h$-Sacks property. Therefore, the generalised Sacks forcing is unable to increase the size of $\dstar \pow$, but it is shown to be possible to increase the size of $\dstar\id$ by using an iteration or a product of the forcing.

In this text we will answer an open question from \cite{BBFM} and prove that there exist $\kappa$ many functions $h_\xi\in\gbs$ and cardinals $\lambda_\xi>\kappa$ such that it is simultaneously consistent that $\dstar{h_\xi}=\lambda_\xi$ for all $\xi<\kappa$. The strategy will be the same, in that we consider a product of Sacks-like forcings that have the $h_\xi$-Sacks property, but not the $h_\eta$-Sacks property for different $\xi$ and $\eta$. In the first section, we introduce the Sacks-like forcing $\forcing h$ and prove that it preserves cardinals and cofinalities and use fusion to show that it satisfies certain Sacks properties. In the second section we consider products of such forcings, we show that properties such as the preservation of cardinals and cofinalities and the relevant Sacks properties are preserved under ${\leq}\kappa$-support products and we use this to prove the consistency of our result.

For the sake of brevity, from now on we will assume that $\kappa$ denotes a strongly inaccessible cardinal. We will also fix the convention that $h,H,F\in\gbs$ denote increasing cofinal cardinal functions (i.e. $\ran(h)$ is cofinal in $\kappa$ and $h(\alpha)$ is a cardinal for each $\alpha<\kappa$). This convention extends to $h_\xi$, $F_0$, and other subscripts.

\section*{The Forcing}

Let us establish some notation to discuss trees on $\gbs$ before we define our  forcing notion.

Let $T\subset{}^{<\kappa}\kappa$ be a tree, then we let $[T]$ denote the set of branches of $T$. Usually $[T]\subset\gbs$. We call $u\in T$ a {\bf non-splitting} node (of $T$) if there is a unique $\beta<\kappa$ such that $u^\frown\beta\in T$ and otherwise we call $u$ an {\bf $\alpha$-splitting} node (of $T$), where $\alpha\leq\card{\st{\beta<\kappa\mid u^\frown\beta\in T}}$. We say $u$ is a {\bf splitting} node if it is a 2-splitting node. If $u$ is an $\alpha$-splitting node, but not a $\xi$-splitting node for any $\xi$ with $|\xi|>|\alpha|$, then we say that $u$ is a {\bf sharp} $\alpha$-splitting node. We let $u\in \Split_\alpha(T)$ iff $u$ is splitting and $\ot(\st{\beta<\ot(u)\mid u\restriction \beta\text{ is splitting}})=\alpha$. If $u\in T$ is splitting, let $\suc(u,T)=\st{v\in T\mid \exists \beta<\kappa(v=u^\frown\beta)}$.

\begin{dfn}\label{forcing definition}
The conditions of the forcing $\forcing{h}$ are trees $T\subset \gfbs$ that satisfy the following properties:
\begin{enumerate}[label=(\roman*)]
\item for any $u\in T$ there exists $v\in T$ such that $u\subset v$ and $v$ is splitting,
\item if $u\in \Split_\alpha(T)$, then $u$ is an $h(\alpha)$-splitting node in $T$,
\item if $\ab{u_\alpha\mid \alpha<\gamma}$ is an increasing sequence of splitting nodes, then $\Cup_{\alpha<\gamma}u_\alpha$ is splitting. 
\end{enumerate}
The order is defined as $T\leq_{\forcing h} S$ (here $T$ is stronger than $S$) iff $T\subset S$ and for every $\alpha<\kappa$ and $u\in T$ a splitting node, if $\suc(u,T)\neq \suc(u,S)$, then $|\suc(u,T)|<|\suc(u,S)|$. If the forcing notion is clear from context, we will write $T\leq S$ in place of $T\leq_{\forcing h}S$
\end{dfn}

Note that conditions (i) and (iii) imply that $[T]\subset\gbs$ and that the set of splitting nodes is club in $T$. We naturally want $\forcing{h}$ to preserve cardinalities. If we assume that $\b V\md\ap{2^\kappa=\kappa^+}$, then it is clear that $\forcing{h}$ has the ${<}\kappa^{++}$-chain condition, since $|\forcing{h}|\leq |\c P(\gfbs)|=2^\kappa=\kappa^+$, where the former equality is implied by $\kappa^{<\kappa}=\kappa$, which in turn follows from $\kappa$ being inaccessible. It is therefore clear that cardinalities above $\kappa^+$ are preserved under assumption of $\b V\md\ap{2^\kappa=\kappa^+}$.

To preserve cardinalities less than or equal to $\kappa$, we show that $\forcing{h}$ is ${<}\kappa$-closed.

\begin{lmm}\label{closure}
$\forcing{h}$ is ${<}\kappa$-closed. That is, for any $\lambda<\kappa$, if $\ab{T_\xi\mid \xi<\lambda}$ is a descending chain of conditions, then there exists a condition $T$ such that $T\leq T_\xi$ for all $\xi<\lambda$.
\end{lmm}
\begin{proof}
The condition $T$ that is below all $T_\xi$ will be simply $T=\Cap_{\xi}T_\xi$. Clearly $T$ is a tree as well. We have to show that $T\in\forcing{h}$, which we do by veryfing points (i), (ii) and (iii) from \Cref{forcing definition}. We will make use of the following claim.

{\bf Claim.}\quad If $u\in T=\Cap T_\xi$, then there is $\eta<\lambda$ such that $\suc(u,T)=\suc(u,T_\xi)$ for all $\xi\in[\eta,\lambda)$.

\begin{subproof}
Suppose that $u\in T$, and let $\lambda_\xi=|\suc(u,T_{\xi})|$, then the ordering on $\forcing{h}$ dictates that $\ab{\lambda_\xi\mid \xi<\lambda}$ is a descending sequence of cardinals, hence there is $\eta<\lambda$ such that $\lambda_\xi=\lambda_\eta$ for all $\xi\in[\eta,\lambda)$. But then $\suc(u,T_\xi)=\suc(u,T_\eta)$ for all $\xi\in[\eta,\lambda)$ by the ordering of $\forcing{h}$.
\end{subproof}

(i)\quad Let $u\in T$, and let $f\in[T]$ be a branch for which $u\subset f$. If $\ot(f)<\kappa$, then $f\in T_\xi$ for each $\xi$, thus by the claim there is some $\xi<\lambda$ for which $\suc(f,T_\xi)=\suc(f,T)=\emp$. Then clearly $T_\xi\notin\forcing{h}$, which is a contradiction, hence $\ot(f)=\kappa$. Let $S_\xi=\st{\alpha\in[\ot(u),\kappa)\mid f\restriction \alpha\text{ is splitting in }T_\xi}$, then since $T_\xi\in\forcing{h}$ satisfies (i) and (iii), we see that $S_\xi$ is a club set. But then $\Cap_{\xi<\lambda}S_\xi$ is club. Any $v\in \Cap_{\xi<\lambda}S_\xi$ is splitting in all $T_\xi$, thus by the claim it is splitting in $T$, and by definition of $S_\xi$ it follows that $u\subset v$.

(ii)\quad If $u\in\Split_\alpha(T)$, then by the claim there is $\xi<\lambda$ such that $\suc(u,T_\xi)=\suc(u,T)$. If $v\subset u$ is splitting in $T$, then $v$ is splitting in $T_\xi$ as well, therefore $u\in\Split_\beta(T_\xi)$ for some $\beta\geq \alpha$. This shows that $u$ is $h(\beta)$-splitting in $T_\xi$, hence it is $h(\beta)$-splitting in $T$ and by $h(\alpha)\leq h(\beta)$ we see that $u$ is also $h(\alpha)$-splitting in $T$.

(iii)\quad Let $\ab{u_\alpha\mid \alpha<\gamma}$ be an increasing sequence of splitting nodes in $T$, then for every $\xi<\lambda$ we also see that $\ab{u_\alpha\mid \alpha<\gamma}$ is a sequence of splitting nodes in $T_\xi$, and thus $u=\Cup_{\alpha<\gamma}u_\alpha$ is a splitting node in all $T_\xi$, hence by the claim $u$ is splitting in $T$.
\end{proof}

\begin{crl}
$\forcing{h}$ preserves all cardinalities and cofinalities $\leq\kappa$.\renewcommand{\qedsymbol}{$\square$}\qedhere
\end{crl}

What is left, is to show that $\kappa^+$ is also preserved. This will be a consequence of showing that $\forcing{h}$ has the $F$-Sacks property for some suitably large $F\in\gbs$: if $\dot f$ is a $\forcing{h}$-name and $T\fc\ap{\dot f:\check\kappa\to \check\kappa^+}$ for some $T\in\forcing{h}$, then using (the proof of) the $F$-Sacks property we may produce sets $A_\xi$ with $|A_\xi|=F(\xi)$ for each $\xi<\kappa$ such that $T'\fc\ap{\dot f(\check\xi)\in \check A_\xi}$ for some stronger $T'\leq T$, and thus $\dot f$ is forced to have a range contained in $\Cup_\xi A_\xi$ and cannot be cofinal in $\kappa^+$.

Before we can prove that $\forcing{h}$ indeed has the $F$-Sacks property for some suitably large $F$, we will need to show that $\forcing{h}$ is closed under fusion. It will be helpful to establish some notation for subtrees. If $u\in T$, let $T_u=\st{v\in T\mid u\subset v\lo v\subset u}$. It is clear that $T_u\leq T$. 

Let a tree $T\in \forcing{h}$ be called \emph{sharp} if every $u\in\Split_\alpha(T)$ is a sharp $h(\alpha)$-splitting node. It is clear that by pruning we may find a sharp $T^*$ below any condition $T\in\forcing{h}$ such that $\Split_\alpha(T^*)\subset \Split_\alpha(T)$ for every $\alpha<\kappa$, and we may assume that we can canonically do so, thus we will hereby fix the notation $T^*$ to denote a sharp tree below condition $T$. We will write $(\forcing{h})^*=\st{T\in \forcing{h}\mid T\text{ is sharp}}$, which embeds densely into $\forcing{h}$. 

For $T,S\in\forcing{h}$, we let $T\leq_\alpha S$ iff $T\leq S$ and $\Split_\alpha(T)=\Split_\alpha(S)$. A {\bf fusion sequence} is a sequence $\ab{T_\alpha\in\forcing{h}\mid\alpha<\kappa}$ such that $T_{\beta}\leq_\alpha T_\alpha$ for all $\alpha\leq\beta<\kappa$.

\begin{lmm}
If $\ab{T_\alpha\in\forcing{h}\mid \alpha<\kappa}$ is a fusion sequence, then $T=\Cap T_\alpha\in \forcing{h}$.
\end{lmm}
\begin{proof}
Clearly $T$ is a tree on ${}^{<\kappa}\kappa$. We check conditions (i), (ii) and (iii) of \Cref{forcing definition}.

(i)\quad Let $u\in T$ and $\alpha=\ot(u)$, then for any $\beta,\beta'\geq\alpha$ we see that $\Split_\alpha(T_\beta)=\Split_\alpha(T_{\beta'})$, thus pick $\beta>\alpha$ arbitrarily. Since $\alpha=\ot(u)$, necessarily there exists some $v\in \Split_\alpha(T_\beta)$ such that $u\subset v$, and since $\beta>\alpha$ we see that $v\in\Split_\alpha(T)$.

(ii)\quad Let $s\in\Split_\alpha(T)$, then $s$ is $h(\alpha)$-splitting in $T_{\alpha+1}$. Let $\lambda_s=|\suc(s,T_{\alpha+1})|\geq h(\alpha)$ and let $\ab{t_\xi\mid \xi<\lambda_s}$ enumerate those $t\supset s$ such that $t\in\Split_{\alpha+1}(T_{\alpha+1})$. For all $\beta,\beta'\geq \alpha+1$ we have $\Split_{\alpha+1}(T_\beta)=\Split_{\alpha+1}(T_{\beta'})$, therefore for each $\xi<\lambda_s$ we see that $t_\xi\in T_\beta$ for all $\beta>\alpha$, thus $t_\xi\in T$. Therefore $s$ is $h(\alpha)$-splitting in $T$.

(iii)\quad Let $\gamma$ be limit and let $\ab{u_\alpha\in\Split_\alpha(T)\mid \alpha<\gamma}$ such that $u_\alpha\subset u_{\alpha'}$ for each $\alpha<\alpha'$, then $u_\alpha\in\Split_\alpha(T_\gamma)$ and thus $u=\Cup u_\alpha$ is an $h(\gamma)$-splitting node in $T_\gamma$, hence $u\in \Split_\gamma(T_\gamma)$. It follows that $u\in\Split_\gamma(T_\beta)$ for all $\beta\geq\gamma$, thus $u\in\Split_\gamma(T)$.
\end{proof}

We are now ready to prove the two main ingredients necessary for separating the localisation cardinals. We will show that for any $h$ there is some faster growing  $F$ such that $\forcing{h}$ has the $F$-Sacks property, and reversely for any $F$ there exists some faster growing  $h$ such that $\forcing{h}$ does not have the $F$-Sacks property. In other words, for any $F_0$ we may find $h$ and $F_1$ such that $\forcing{h}$ does not have the $F_0$-Sacks property, but does have the $F_1$-Sacks property.

\begin{thm}\label{having Sacks}
For any $h$ there exists $F$ such that $h\leq F$ and $\bb S^{h}_\kappa$ has the  $F$-Sacks property.
\end{thm}
\begin{proof}
We will let $F:\alpha\mapsto h(\alpha)^{|\alpha|}$ and show that $\forcing{h}$ has the $F$-Sacks property.

Let $T_0\in\forcing{h}$ and let $\dot f$ be a name such that $T_0\fc\ap{\dot f:\check\kappa\to\check\kappa}$. If $T_0\fc \ap{\dot f\in\check{\b V}}$, then the existence of an appropriate $F$-slalom is obvious, so we assume that $T_0\fc\ap{\dot f\notin \check{\b V}}$. We will construct a fusion sequence $\ab{T_\xi\mid \xi<\kappa}$ and a family of sets $\st{B_\xi\subset \kappa\mid \xi<\kappa}$ with $|B_\xi|\leq F(\xi)$ such that  $\Cap T_\xi=T\fc\ap{\dot f(\check\xi)\in \check B_\xi}$ for each $\xi<\kappa$. Consequently, we can define the $F$-slalom $\phi:\xi\mapsto B_\xi$ to see that $T\fc\ap{\dot f(\check\xi)\in \check\phi(\check\xi)}$ for all $\xi<\kappa$.

In general, we will assume each $T_\xi$ has the following property:
\begin{align*}
(*)\quad\text{for every }u\in\Split_\alpha(T_\xi)\text{ with }\alpha<\xi\text{ we have }|\suc(u,T_\xi)|=h(\alpha).
\end{align*}
This is vacuously true for $T_0$, and by using sharp trees at successor stages of our construction, $(*)$ will follow by induction. If $\gamma$ is limit, we will let $T_\gamma=\Cap_{\xi<\gamma}T_\xi$, which will be a condition by the proof of \Cref{closure}. $T_\gamma$ need not necessarily be a sharp tree, but it is at least sharp for all splitting levels less than $\gamma$, which is enough for $(*)$.

Suppose $T_\xi$ has been defined, then we will define $T_{\xi+1}$ that limits the possible values of $\dot f(\xi)$ and such that $T_{\xi+1}\leq_\xi T_\xi$. First note that if $T_\xi$ has property $(*)$, then $T_\xi^*\leq_\xi T_\xi$: If $u$ is splitting in $T_\xi$ and $u\notin T_\xi^*$, then $u$ was removed because there is some $v\subset u$ such that $\suc(v,T_\xi)$ is too large for sharpness. But then by $(*)$ it follows that $v\in\Split_\alpha(T_\xi)$ for some $\alpha\geq \xi$, hence $u\in\Split_\beta(T_\xi)$ for some $\beta>\xi$.

Let $V_\xi =\Cup\{\suc(u,T_\xi^*)\mid u\in\Split_\xi(T_\xi^*)\}$, and for each $v\in V_\xi$ find a condition $T^v\leq (T_\xi^*)_v$ such that $T^v\fc\ap{\dot f(\check\xi)=\check\beta_\xi^v}$ for some $\beta_\xi^v<\kappa$. Let $u\in\Split_\xi(T^v  )$ and $w\in\suc(u,T^v  )$ be arbitrary and consider the subtree  $T^v_{w}$ of $T^v  $ generated by the initial segment $w$. We let $G_\xi:V_\xi \to \c P(T_\xi)$ send $v\mapsto T^v_w$. Note that the $\alpha$-th splitting level of $G_\xi(v)=T^v_w$ corresponds to the $(\xi+1+\alpha)$-th splitting level of $T^v$. 

Now we define:
\begin{align*}
T_{\xi+1}&=\Cup G_\xi[V_\xi ]=\Cup \st{G_\xi(v)\mid v\in V_\xi },\\
B_{\xi}&=\st{\beta_\xi^v\mid v\in V_\xi }.
\end{align*}
For each $v\in V_\xi $ we have $v\in G_\xi(v)$, thus each successor of a splitting node in $\Split_\xi(T_\xi)$ is in $T_{\xi+1}$, therefore we see that $\Split_{\xi}(T_{\xi+1})=\Split_\xi(T_\xi)$. If $u\in\Split_{\xi+1+\alpha} (T_{\xi+1})$ for some $\alpha<\kappa$, then $u\in\Split_{\alpha}(G_\xi(v))$, thus $u\in \Split_{\xi+1+\alpha}(T^v)$, and since $T^v \in\forcing{h}$, we see that $u$ is $h(\xi+1+\alpha)$-splitting. Therefore $T_{\xi+1}$ satisfies (ii) of \Cref{forcing definition}. It is easy to check (i) and (iii), thus we can conclude that $T_{\xi+1}\in\forcing{h}$ and that $T_{\xi+1}\leq_\xi T_\xi$.

Note that $|B_\xi|\leq|V_\xi |=|\Split_\xi(T_{\xi}^*)|\cdot h(\xi)\leq h(\xi)^{|\xi|}=F(\xi)$, as intended. 

Now $T^v \fc \ap{\dot f(\check\xi)\in \check B_\xi}$, and $\st{T^v \mid v\in V_\xi}$ is predense below $T_{\xi+1}$, thus $T_{\xi+1}\fc\ap{\dot f(\check\xi)\in \check B_\xi}$. Let $T=\Cap_{\xi<\kappa}T_\xi$, then by the fusion lemma $T\in\forcing{h}$, and $T\fc\ap{\dot f(\check \xi)\in \check B_\xi}$ for all $\xi<\kappa$.
\end{proof}

\begin{crl}
$\forcing{h}$ preserves $\kappa^+$.\renewcommand{\qedsymbol}{$\square$}\qedhere
\end{crl}

Before we prove the existence of $F$ such that $\forcing h$ does not have the $F$-Sacks property, we will state a lemma that will be useful later on as well.

\begin{lmm}\label{splitting nodes club}
Let $T\in \forcing h$ and let $C_T=\st{\alpha<\kappa\mid \Split_\alpha(T)=T\cap {}^\alpha\kappa}$, then $C_T$ is a club set.
\end{lmm}
\begin{proof}
For $\alpha_0\in\kappa$ we can recursively define $\alpha_{n+1}$ large  enough such that $\Split_{\alpha_n}(T)\subset {}^{\leq \alpha_{n+1}}\kappa$ for each $n\in\omega$. Let $\alpha=\Cup_{n\in\omega}\alpha_n$, then $\alpha\in C_T$, hence $C_T$ is unbounded. It is easy to see that $C_T$ is continuous.
\end{proof}

\begin{thm}\label{not having Sacks} 
Let $F\in{}^\kappa\kappa$, then there exists $h$ such that $\bb S^{h}_\kappa$ does not have the $F$-Sacks property.
\end{thm}

\begin{proof}
Let $h$ be such that $F(\alpha)<h(\alpha)$ for all $\alpha\in S$, where $S$ is a stationary subset of $\kappa$. We will show that $\bb S^h_\kappa$ does not have the $F$-Sacks property.

Let $\dot f$ be a name for the generic $\bb S^h_\kappa$-real in $\gbs$, let $\phi$ be an $F$-slalom, let $T\in\bb S^h_\kappa$ and let $\alpha_0<\kappa$. We want to find some $\alpha\geq\alpha_0$ and $S\leq T$ such that $S\fc\ap{\dot f(\check\alpha)\notin\check\phi(\check\alpha)}$. If we can find $u\in T\cap{}^{\alpha+1}\kappa$ such that $u(\alpha)\notin \phi(\alpha)$, then $T_u$ will be sufficient.

Let $C_T$ be as defined in \Cref{splitting nodes club} and $\alpha\in C_T\cap S$ such that $\alpha_0\leq\alpha$, then $\Split_\alpha(T)=T\cap {}^\alpha\kappa$, thus each $t\in T\cap {}^\alpha\kappa$ is an $h(\alpha)$-splitting node. Hence, there is a set $X\subset\kappa$ with $|X|=h(\alpha)$ such that $t^\frown\gamma\in T$ for all $\gamma\in X$. Since $|\phi(\alpha)|=F(\alpha)<h(\alpha)$, there is some $\gamma\in X$ such that $\gamma\notin \phi(\alpha)$, and thus $u=t^\frown\gamma$ is as desired.
\end{proof}

\section*{Products}

We see that for any $F_0$, we can find a faster growing $F_1$ and some suitable $h$ such that the forcing $\forcing{h}$ has the $F_1$-Sacks property and not the $F_0$-Sacks property, thus forcing with $\forcing{h}$ will not increase $\dstar{F_1}$, but has the potential to increase $\dstar{F_0}$. 

In order to increase $\dstar{F_0}$ we will need to add many $\forcing{h}$-generic $\kappa$-reals to the ground model. This can be either done with an iteration, or with a product. Iteration has the drawback that once we have forced $2^\kappa$ to be of size $\kappa^{++}$, the forcing $\forcing{h}$ no longer has the ${<}\kappa^{++}$-cc, and thus we cannot sufficiently control the iteration past this point. We will therefore focus on the product instead.

Let $A$ be a set of ordinals and $h_\xi$ for each $\xi\in A$. Let $\c C$ be the set of (choice) functions $p:A\to\bb \Cup_{\xi\in A}\forcing{h_\xi}$ such that $p(\xi)\in \forcing{h_\xi}$ for each $\xi\in A$. For any $p\in \c C$, we define the {\bf support} of $p$ as $\supp(p)=\st{\xi\in A\smid p(\xi)\neq\ft_{\forcing{h_\xi}}}$. We define the ${\leq}\kappa$-support product as follows: 
\begin{align*}
\textstyle\prod_{\xi\in A}\forcing{h_\xi}=\st{p\in\c C\smid \card{\supp(p)}\leq \kappa}.
\end{align*}
We will fix $A$, $\ab{h_\xi\mid \xi\in A}$ and $\c C$ and use the shorthand $\bar{\bb S}=\prod_{\xi\in A}\forcing{h_\xi}$, unless stated otherwise. If $p,q\in\bar{\bb S}$, then $q\leq_{\bar{\bb S}} p$ iff $q(\xi)\leq_{\forcing{h_\xi}}p(\xi)$ for all $\xi\in A$. We will again write $q\leq p$ instead of $q\leq_{\bar{\bb S}} q$ if the context is clear.

If $\ab{p_\alpha\mid \alpha<\gamma}$ is a sequence of conditions in $\bar{\bb S}$ such that $p_{\alpha'}\leq_{\bar{\bb S}} p_\alpha$ for all $\alpha\leq \alpha'$, then we let $\La_{\alpha<\gamma}p_\alpha:A\to \Cup_{\xi\in A}\forcing{h_\xi}$ be the function sending $\xi\mapsto \Cap_{\alpha<\gamma}p_\alpha(\xi)$. If we assume that $\gamma<\kappa$, then by the ${<}\kappa$-closure of each $\forcing{h_\xi}$ (\Cref{closure}) it follows that $\Cap_{\alpha<\gamma} p_\alpha(\xi)\in \forcing{h_\xi}$, and thus $\La_{\alpha<\gamma} p_\alpha\in \c C$. It is easy to see that $|\supp(\La_{\alpha<\gamma} p_\alpha)|\leq \kappa$ as well, thus we see that $\bar{\bb S}$ is ${<}\kappa$-closed:

\begin{lmm}\label{closure product}
$\bar{\bb S}$ is ${<}\kappa$-closed.\renewcommand{\qedsymbol}{$\square$}\qedhere
\end{lmm}

We will also need a generalisation of the fusion lemma to work on product forcings. The generalisation of fusion described here is analogous to what is described in \cite{Kan80} or \cite{BBFM}.  

Given $p,q\in\bar{\bb S}$, $\alpha<\kappa$, and $Z\subset A$ with $|Z|<\kappa$, let $q\leq_{Z,\alpha}p$ iff $q\leq p$ and for each $\xi\in Z$ we have $q(\xi)\leq_\alpha p(\xi)$.

A {\bf generalised fusion sequence} is a sequence $\ab{(p_\alpha,Z_\alpha)\mid \alpha<\kappa}$ such that:

\begin{itemize}
\item $p_\alpha\in \bar {\bb S}$ and $Z_\alpha\in [A]^{<\kappa}$ for each $\alpha<\kappa$,
\item $p_{\beta}\leq_{Z_\alpha,\alpha}p_\alpha$ and $Z_\alpha\subset Z_{\beta}$ for all $\alpha\leq\beta<\kappa$,
\item for limit $\delta$ we have $Z_\delta=\Cup_{\alpha<\delta} Z_\alpha$,
\item $\Cup_{\alpha<\kappa} Z_\alpha=\Cup_{\alpha<\kappa}\supp(p_\alpha)$.
\end{itemize}

\begin{lmm}
If $\ab{(p_\alpha,Z_\alpha)\mid \alpha<\kappa}$ is a generalised fusion sequence, then $\La_{\alpha<\kappa}p_\alpha\in \bar{\bb S}$.
\end{lmm}
\begin{proof}
Suppose that $\ab{(p_\alpha,Z_\alpha)\mid \alpha<\kappa}$ is a generalised fusion sequence, and let $p=\La_{\alpha<\kappa}p_\alpha$. The last condition of the definition above implies that for any $\xi\in\supp(p)$ there is some $\eta_\xi\in\kappa$ such that $\xi\in Z_{\eta_\xi}$. This means that if $\beta\geq\alpha\geq\eta_\xi$, then $p_\beta(\xi)\leq_\alpha p_\alpha(\xi)$, and thus $\ab{p_\alpha(\xi)\mid \alpha>\eta_\xi}$ is a fusion sequence in $\forcing{h_\xi}$. Since $\forcing{h_\xi}$ is closed under fusion sequences, we therefore conclude that $p(\xi)=\Cap_{\alpha<\kappa} p_\alpha(\xi)\in\forcing{h_\xi}$. Since $\supp(p)=\Cup_{\alpha<\kappa} Z_\alpha$, we see that $|\supp(p)|\leq\kappa$, thus we can conclude that $p\in \bar{\bb S}$.
\end{proof}

By \Cref{closure product}, $\bar{\bb S}$ preserves all cardinalities up to and including $\kappa$. Suppose that each $\forcing{h_\xi}$ has the $F$-Sacks property for some suitably large $F$. We will show in the next lemma that this implies that $\bar{\bb S}$ has the $F$-Sacks property and therefore preserves $\kappa^+$. Finally, if we assume that $\b V\md\ap{2^\kappa=\kappa^+}$, then a standard $\Delta$-system argument (see e.g. \cite{Jech} Lemma 15.4) shows that $\bar{\bb S}$ is ${<}\kappa^{++}$-cc as well. Thus, $\bar{\bb S}$ preserves all cardinals and cofinalities assuming that each $\forcing{h_\xi}$ has the $F$-Sacks property for some fixed $F\in\gbs$.

Before we prove the lemma, let us introduce some notation related to the product of forcings. Suppose $\bar {\bb P}=\prod_{\xi\in A}\bb P_\xi$ is a product with ${\leq}\kappa$-support and $G$ is $\bar{\bb P}$-generic. If $B\subset A$, we define 
\begin{align*}
\bar{\bb P}\restriction B&=\st{p\restriction B\mid p\in \bar{\bb P}}&G\restriction B&=\st{p\restriction B\mid p\in G}
\end{align*}
Let $B^c=A\setminus B$, then clearly $\bar{\bb P}$ and $(\bar{\bb P}\restriction B)\times (\bar{\bb P}\restriction B^c)$ are forcing equivalent, $(G\restriction B)\times (G\restriction B^c)$ is $(\bar{\bb P}\restriction B)\times (\bar{\bb P}\restriction B^c)$-generic and $\b V[G]=\b V[(G\restriction B)\times (G\restriction B^c)]=\b V[G\restriction B][G\restriction B^c]$.

\begin{lmm}\label{having sacks product}
Let $B\subset A$ be sets of ordinals and $B^c=A\setminus B$, and consider a sequence of functions $\ab{h_\xi\mid \xi\in A}$. We define the ${\leq}\kappa$-support product $\bar {\bb S}=\prod_{\xi\in A}\forcing{h_\xi}$ and we assume $G$ is an $\bar{\bb S}$-generic filter. If $F:\alpha\mapsto (\sup_{\xi\in B^c}h_\xi(\alpha))^{|\alpha|}$ is a well-defined function in $\gbs$, then for each $f\in (\gbs)^{\b V[G]}$ there is $\phi\in(\Loc_F)^{\b V[G\restriction B]}$ such that $f\in^*\phi$.
\end{lmm}
\begin{proof}
Let $\dot f$ be a name such that $\fc_{\bar{\bb S}}\ap{\dot f:\check\kappa\to\check\kappa}$ and let $p\in\bar{\bb S}$, then we will construct a name $\dot \phi$ and a condition $p'\leq p$ such that $p'\fc\ap{\dot\phi\in(\Loc_{\check F})^{\b V[\dot G\restriction\check B]}}$.

The proof is essentially the same as the proof of \Cref{having Sacks}, except that we work with generalised fusion sequences and have to construct a name $\dot\phi$ for the appropriate $F$-slalom, since such slalom is not generally present in the ground model. That is, we will construct a sequence $\ab{(p_\xi,Z_\xi)\mid \xi<\kappa}$ with each $p_\xi\in \bar{\bb S}$ that is a generalised fusion sequence in $\bar{\bb S}$ and names $\dot D_\xi$ for sets of ordinals $D_\xi\in \b V[G\restriction B]$ with $|D_\xi|\leq F(\xi)$, such that $p_{\xi+1}\fc\ap{\dot f(\check\xi)\in\dot D_\xi}$. Furthermore, for each $\xi<\kappa$ and $\beta\in Z_\xi$ we will make sure that $p_\xi(\beta)\in(\forcing{h_\beta})^*$ is sharp. To start, we let $p_0= p$ and we let $Z_0=\emp$. At limit stages $\delta$ we can define $p_\delta'=\La_{\xi<\delta}p_\xi$ and let $p_\delta\leq p_\delta'$ be defined elementwise such that $p_\delta(\beta)=(p_\delta'(\beta))^*$ is sharp for each $\beta\in Z_\delta$.

Suppose we have defined $p_\xi\in \bar{\bb S}$ and $Z_\xi$ and that $|Z_\xi|\leq |\xi|$, then for each $\beta\in Z_\xi$ we define 
\begin{align*}
V^\beta_\xi=\Cup\st{\suc(u,p_\xi(\beta))\smid u\in \Split_\xi(p_\xi(\beta))}
\end{align*}
Let $\c V_\xi$ be the set of choice functions on $\{V_\xi^\beta\mid \beta\in Z_\xi\}$ and $\c V_\xi'$ the set of choice functions on $\{V_\xi^\beta\mid \beta\in Z_\xi\setminus B\}$. By induction hypothesis $p_\xi(\beta)\in(\forcing{h_\beta})^*$ for each $\beta\in Z_\xi\setminus B$, hence we know that $|\Split_\xi(p_\xi(\beta))|\leq h_\beta(\xi)^{|\xi|}\leq F(\xi)$, and thus $|V_\xi^\beta|\leq F(\xi)$ for all $\beta\in Z_\xi\setminus B$. Since we assume that $|Z_\xi|\leq |\xi|$, we therefore have $|\c V_\xi'|\leq F(\xi)^{|\xi|}=F(\xi)$ (assuming without loss of generality that $F(\xi)$ is infinite). For any choice function $g\in \c V_\xi$, let $(p_\xi)_g$ be the condition with $(p_\xi)_g(\beta)=p_\xi(\beta)$ if $\beta\notin Z_\xi$ and  $(p_\xi)_{g}(\beta)=({p_\xi(\beta)})_{g(\beta)}$ if $\beta\in Z_\xi$. Here $({p_\xi(\beta)})_{g(\beta)}$ is the subtree of $p_\xi(\beta)$ generated by the initial segment $g(\beta)$.

Let $\zeta=|\c V_\xi|$ then $\zeta<\kappa$ by inaccessibility of $\kappa$. Fix some enumeration $\ab{g_\eta\mid \eta<\zeta}$ of $\c V_\xi$, which we will use to define a decreasing sequence of conditions $r_\eta$ with $r_\eta\leq_{Z_\xi,\xi}p_\xi$ for each $\eta<\zeta$. Let $r_0=p_\xi$. For limit $\delta<\zeta$ let $r_\delta=\La_{\eta<\delta}r_\eta$, which is a condition by ${<}\kappa$-closure (\Cref{closure product}). Assuming that $r_\eta\leq_{Z_\xi,\xi}p_\xi$ for each $\eta<\delta$, it is easy to see that $r_\delta\leq_{Z_\xi,\xi}p_\xi$ as well.

Suppose $r_\eta$ is defined and $r_\eta\leq_{Z_\xi,\xi} p_\xi$, then in particular $r_\eta(\beta)\leq_\xi p_\xi(\beta)$ for all $\beta\in Z_\xi$, and thus $\Split_\xi(r_\eta(\beta))=\Split_\xi(p_\xi(\beta))$ for all $\beta\in Z_\xi$, and therefore by definition of the ordering on $\forcing{h_\beta}$ and the fact that $p_\xi(\beta)$ is sharp, we see that $V_\xi^\beta$ is exactly the set of successors of nodes at the $\xi$-th splitting level of $r_\eta(\beta)$. Take the $\eta$-th choice function $g_\eta\in \c V_\xi$, and let $r'_{\eta}\leq (r_\eta)_{g_\eta}$ be such that $r'_{\eta}\fc\ap{\dot f(\check\xi)=\check \beta^\eta_\xi}$ for some ordinal $\beta^\eta_\xi$. We define $r_{\eta+1}$ elementwise.

If $\beta\notin Z_\xi$, then we simply take $r_{\eta+1}(\beta)=r_{\eta}'(\beta)$.

If $\beta\in Z_\xi$, fix some $w\in\suc(u,r'_{\eta}(\beta))$ for some $u\in\Split_\xi(r'_{\eta}(\beta))$ and consider the subtree $(r'_{\eta}(\beta))_w$ generated by the initial segment $w$. Now we are ready to define $r_{\eta+1}(\beta)$ as
\begin{align*}
r_{\eta+1}(\beta)=(r'_{\eta}(\beta))_w\ \cup\ \st{u\in r_{\eta}(\beta)\smid \exists v\in V_\xi^\beta\setminus \{g_\eta(\beta)\}\ (u\subseteq v\text{ or }v\subset u)}.
\end{align*}
In words, $r_{\eta+1}(\beta)$ is the result of replacing the extensions of $g_\eta(\beta)\in r_\eta(\beta)$ by $(r'_\eta(\beta))_w$ that decides $\dot f$ at $\xi$, where the function of taking the subtree of $r'_\eta(\beta)$ generated by $w$ is to make sure $r_{\eta+1}(\beta)$ has enough successors at each splitting level to be in $\forcing{h_\beta}$.

To finish the construction of the next condition in the fusion sequence, we use ${<}\kappa$-closure to define $p_{\xi+1}'=\La_{\eta<\zeta}\ r_\eta$ and let $p_{\xi+1}=(p'_{\xi+1})^*$ be sharp. To see that $p_{\xi+1}\leq_{Z_\xi,\xi}p_\xi$, note that for every $\beta\in Z_\xi$ and $v\in V^\beta_\xi$ we have $v\in r_\eta(\beta)$ for all $\eta<\zeta$, hence $v\in p_{\xi+1}(\beta)$. This implies by definition of $V^\beta_\xi$ that $p_{\xi+1}(\beta)\leq_\xi p_\xi(\beta)$ for all $\beta\in Z_\xi$. Finally, we can let $Z_{\xi+1}=Z_\xi\cup\st{\delta}$ for some ordinal $\delta$, using bookkeeping to make sure that $\Cup_{\xi<\kappa}Z_\xi=\Cup_{\xi<\kappa}\supp(p_\xi)$.

Note that the set of conditions $r\leq p_{\xi+1}$ such that $|r(\beta)\cap V_\xi^\beta|=1$ for all $\beta\in Z_\xi$, is dense below $p_{\xi+1}$. For any such $r$, let $g$ map $\beta$ to the unique element of $r(\beta)\cap V_\xi^\beta$ for each $\beta\in Z_\xi$, then $g\in \c V_\xi$, so we see that there exists $\eta<\zeta$ such that $g=g_\eta$. We will show that $r\leq r_\eta'$, which implies that $r\fc\ap{\dot f(\check \xi)=\check\beta_\xi^\eta}$.

For any $\beta$ we have $r(\beta)\leq p_{\xi+1}(\beta)\leq r_{\eta+1}(\beta)$. If $\beta\notin Z_\xi$, then we simply have $r_{\eta+1}(\beta)=r_\eta'(\beta)$, thus we are done. Otherwise $\beta\in Z_\xi$, and we know that $g(\beta)$ is an initial segment of the stem of $r(\beta)$, hence $r(\beta)=(r(\beta))_{g(\beta)}\subset (r_{\eta+1}(\beta))_{g(\beta)}= (r'_\eta(\beta))_w$, where $w$ is as in the definition of $r_{\eta+1}(\beta)$ above. Since $r(\beta)\leq r_{\eta+1}(\beta)$, we also have $r(\beta)=(r(\beta))_w\leq(r_{\eta+1}(\beta))_w=(r'_\eta(\beta))_w$, and since $(r'_\eta(\beta))_w\leq r'_\eta(\beta)$, we see that $r(\beta)\leq r'_\eta(\beta)$.

We now construct the names $\dot D_\xi$ such that:
\begin{align*}
p_{\xi+1}\fc\ap{\dot f(\check\xi)\in \dot D_\xi\text{ and }\dot D_\xi\in \b V[\dot G\restriction\check B]\text{ and }|\dot D_\xi|\leq \check F(\check\xi)}.
\end{align*}
For any $g\in\c V_\xi$, let $g''=g\restriction (Z_\xi\cap B)$, and let $E_g=\st{\eta<\zeta\mid \exists g'\in \c V_\xi'(g'\cup g''=g_\eta)}$ and  $D_\xi^g=\st{\beta_\xi^\eta\mid \eta\in E_g}$. Since $|\c V_\xi'|\leq F(\xi)$, we see that $|E_g|\leq F(\xi)$, hence $|D_\xi^g|\leq F(\xi)$. Clearly, if $g,\tilde g\in \c V_\xi$ and $g\restriction (Z_\xi\cap B)=\tilde g\restriction (Z_\xi\cap B)$, then $D_\xi^g=D_\xi^{\tilde g}$.

Let $\c A_\xi$ be an antichain below $p_{\xi+1}$ such that $r\in \c A_\xi$ implies $|r(\beta)\cap V_\xi^\beta|=1$ for all $\beta\in Z_\xi$, and let $g_r\in \c V_\xi$ be such that $g_r(\beta)$ is the single element of $r(\beta)\cap V_\xi^\beta$ for each $\beta\in Z_\xi$. We define the name $\dot D_\xi=\st{(r,\check D_\xi^{g_r})\mid r\in \c A_\xi}$. It is clear by the above that for each $r\in \c A$ and $\eta$ such that $g_r=g_\eta$ we have $r\fc \ap{\dot f(\check\xi)=\check \beta_\xi^\eta\in \check D_\xi^{g_r}\text{ and }|\check D_\xi^{g_r}|\leq\check F(\check \xi)}$, so by denseness $p_{\xi+1}\fc\ap{\dot f(\check\xi)\in \dot D_\xi\text{ and }|\dot D_\xi|\leq \check F(\check \xi)}$. 

To see that $p_{\xi+1}\fc\ap{ \dot D_\xi\in \b V[\dot G\restriction\check B]}$, we argue within $\b V[ G\restriction B]$. For every $r,\tilde r\in \c A_\xi$ such that $r\restriction B\in G\restriction B$ and $\tilde r\restriction B\in G\restriction B$ we see that the corresponding $g_r$ and $g_{\tilde r}$ have the property that $g_r\restriction (Z_\xi\cap B)=g_{\tilde r}\restriction (Z_\xi\cap B)$, and therefore $D_\xi^{g_r}=D_\xi^{g_{\tilde r}}$. Thus, we can fix any arbitrary such $r\in \c A_\xi$ for which $r\restriction B\in G\restriction B$ holds, and see that $\b V[G\restriction B]\md \ap{p_{\xi+1}\restriction B^c\fc\dot D_\xi =\check D_\xi^{g_r}}$.

Let $p'=\La_{\xi<\kappa}p_\xi$ be the limit of the generalised fusion sequence, and let $\dot\phi$ be a name such that 
$p'\fc\ap{\dot\phi:\check\xi\mapsto \dot D_\xi}$, then $\dot\phi$ names an $F$-slalom in $\b V[G\restriction B]$ and $p'\fc\ap{\dot f\in^*\dot \phi}$.
\end{proof}

If we let $B=\emp$ in the definition of the lemma, then we can simplify this lemma to the following corollary, providing us with the preservation of the Sacks property.

\begin{crl}
\renewcommand{\qedsymbol}{$\square$}
If $\bar{\bb S}=\prod_{\xi\in A}\forcing{h_\xi}$ and each $h_\xi$ is bounded by $h\in\gbs$ and $F:\alpha\mapsto h(\alpha)^{|\alpha|}$, then $\bar{\bb S}$ has the $F$-Sacks property.\qedhere
\end{crl}

Finally the following lemma is based on \Cref{not having Sacks} and shows how we can use products of forcings $\forcing{h_\xi}$ to increase the cardinality of $\dstar{F}$.

\begin{lmm}\label{not having sacks product}
Let $B\subset A$ be sets of ordinals, and consider a sequence of functions $\ab{h_\xi\mid \xi\in A}$. We define the ${\leq}\kappa$-support product $\bar {\bb S}=\prod_{\xi\in A}\forcing{h_\xi}$ and we assume $G$ is a $\bar{\bb S}$-generic filter. Let $\ab{S_\xi\mid \xi\in B}$ be a sequence of stationary sets. If $F$ is such that for each $\xi\in B$ we have $F(\alpha)<h_\xi(\alpha)$ for all $\alpha\in S_\xi$, then $\b V[G]\md\ap{|B|\leq\dstar{F}}$.
\end{lmm}
\begin{proof}
The lemma is trivial if $|B|\leq \kappa^+$, so we will assume that $|B|\geq\kappa^{++}$.

We work in $\b V[G]$. Let $\mu<|B|$ and let $\st{\phi_\xi\mid \xi<\mu}$ be a family of $F$-slaloms, then we want to describe some $f\in\gbs$ such that $f\notin^*\phi_\xi$ for each $\xi<\mu$. Since $\bar{\bb S}$ is ${<}\kappa^{++}$-cc, we could find $A_\xi\subset A$ with $|A_\xi|\leq \kappa^+$ for each $\xi<\mu$ such that $\phi_\xi\in \b V[G\restriction A_\xi]$. Since $|B|>\mu\cdot \kappa^+$, we may fix some $\beta\in B\setminus\Cup_{\xi<\mu}A_\xi$ for the remainder of this proof. Let $f=\Cap_{p\in G}p(\beta)$, then $f\in\gbs$ is the generic $\kappa$-real added by the $\beta$-th term of the product $\bar{\bb S}$. 

Continuing the proof in the ground model, let $\dot f$ be an $\bar {\bb S}$-name for $f$ and $\dot \phi_\xi$ be an $\bar{\bb S}$-name for $\phi_\xi$, let $p\in\bar{\bb S}$ and $\alpha_0<\kappa$. We want to find some $\alpha\geq\alpha_0$ and $q\leq p$ such that $q\fc\ap{\dot f(\check \alpha)\notin \dot \phi_\xi(\check \alpha)}$.

Let $C=\st{\alpha<\kappa\mid p(\beta)\cap {}^\alpha\kappa=\Split_\alpha(p(\beta))}$, which is a club set by \Cref{splitting nodes club}. Since $S_\beta$ is stationary, there exists some $\alpha\geq\alpha_0$ such that $\alpha\in C\cap S_\beta$. Choose some $p_0\leq p$ such that $p_0(\beta)=p(\beta)$ and such that there is a $Y\in[\kappa]^{\leq F(\alpha)}$ for which $p_0\fc\ap{\dot \phi_\xi(\check\alpha)=\check Y}$. This is possible, since $\phi_\xi\in\b V[G\restriction A_\xi]$ and $\beta\notin A_\xi$, therefore we could find $p_0'\in\bar{\bb S}\restriction A_\xi$ with $p_0'\leq p\restriction A_\xi$ and $Y$ with the aforementioned property, and then let $p_0(\eta)=p'_0(\eta)$ if $\eta\in A_\xi$ and $p_0(\eta)=p(\eta)$ otherwise.

Each $t\in p_0(\beta)\cap {}^\alpha\kappa$ is a $h_\beta(\alpha)$-splitting node, hence the set $X=\st{\chi<\kappa\mid t^\frown \chi \in p_0(\beta)}$ has cardinality $|X|\geq h_\beta(\alpha)$. Because $\alpha\in S_\beta$ and $\beta\in B$, we have by our assumptions on $F$ that $|Y|\leq F(\alpha)<h_\xi(\alpha)=|X|$. We can therefore find some $\chi\in X$ such that $\chi\notin Y$. Let $q\leq p_0$ be defined as $q(\beta)=(p_0(\beta))_{t^\frown \chi}$, which is the subtree of $p_0(\beta)$ generated by the initial segment $t^\frown \chi$, and $q(\eta)=p_0(\eta)$ for all $\eta\neq \beta$. Then $q\leq p_0\leq p$ and $q\fc\ap{\dot f(\check \alpha)\notin \dot\phi_\xi(\check\alpha)}$
\end{proof}

\begin{lmm}\label{continuum size}
Let $A$ be a set of ordinals, let $\ab{h_\xi\mid \xi\in A}$ be a sequence of functions, let $\bar{\bb S}=\prod_{\xi\in A}\forcing{h_\xi}$ with $\bar {\bb S}$-generic $G$, and let $F\in\gbs$. Then $\b V\md\ap{2^\kappa=\kappa^+}$ implies $\b V[G]\md\ap{2^\kappa=|\Loc_F|=\kappa^+\cdot |A|}$.
\end{lmm}
\begin{proof}
This is a standard argument.
\end{proof}

We're now ready to use our product forcing to separate many cardinals of the form $\dstar{h}$.
\begin{thm}
There exists a family of functions $\st{g_\eta\mid \eta<\kappa}\subset \gbs$ such that for any $\gamma<\kappa^+$ and any increasing sequence $\ab{\lambda_\xi\mid \xi<\gamma}$ of cardinals with $\lambda_0\geq2^\kappa$ and any $\sigma:\kappa\to\gamma$, there exists a forcing extension in which $\dstar{g_\eta}=\lambda_{\sigma(\eta)}$ for all $\eta<\kappa$.
\end{thm}
\begin{proof}
We assume that $\b V\md\ap{2^\kappa=\kappa^+}$, or otherwise we first use a forcing to collapse $2^\kappa$ to become $\kappa^+$. By a result of Solovay (theorem 8.10 in \cite{Jech}) there exists a family of $\kappa$ many disjoint stationary subsets of $\kappa$, thus let $\st{S_\eta\mid \eta<\kappa}$ be such a family. Let $\kappa\leq \gamma<\kappa^+$ and $\sigma:\kappa\to \gamma$ be given. We will assume without loss of generality that $\sigma$ is bijective, and hence that $\sigma^{-1}:\gamma\to \kappa$ is a well-defined bijection. Let $\ab{\lambda_\xi\mid \xi<\gamma}$ be an increasing sequence of cardinals with $\lambda_0\geq \kappa^+$.

Fix some $F$ such that $F(\alpha)^{|\alpha|}=F(\alpha)$. For each $\eta<\kappa$ we define a function $g_\eta$ as follows:
\begin{align*}
g_\eta(\alpha)=\begin{cases}
F(\alpha)&\text{ if }\alpha\in S_\eta\\
2^{F(\alpha)}&\text{ otherwise}
\end{cases}
\end{align*}
For each $\xi<\gamma$ we define $H_\xi\in \gbs$ as follows:
\begin{align*}
H_\xi(\alpha)=\begin{cases}
F(\alpha)&\text{ if }\alpha\in\Cup_{\xi'<\xi}S_{\sigma^{-1}(\xi')}\\
2^{F(\alpha)}&\text{ otherwise}
\end{cases}
\end{align*}
Note that $g_{\eta}\geq H_{\sigma(\eta)+1}$: if $\alpha\in S_\eta$, then $\alpha\in S_{\sigma^{-1}(\sigma(\eta))}$, so $\alpha\in \Cup_{\xi<\sigma(\eta)+1} S_{\sigma^{-1}(\xi)}$, and thus $H_{\sigma(\eta)+1}(\alpha)=F(\alpha)$. Therefore $\s{ZFC}\ent\ap{\dstar{g_{\eta}}\leq \dstar{H_{\sigma(\eta)+1}}}$.

For each $\xi<\gamma$ let $A_\xi$ be a set of ordinals with $|A_\xi|=\lambda_\xi$, such that $\ab{A_\xi\mid \xi<\gamma}$ is a sequence of mutually disjoint sets, and let $A=\Cup_{\xi<\gamma} A_\xi$. For each $\xi<\gamma$ and $\beta\in A_\xi$, we define $h_\beta\in \gbs$ as $h_\beta=H_{\xi}$. 

The forcing that we will consider is $\bar{\bb S}=\prod_{\beta\in A}\forcing{h_\beta}$ with ${\leq}\kappa$-support. Let $G$ be $\bar{\bb S}$-generic. We will fix some $\xi<\gamma$ for the remainder, and let $B=\Cup_{\xi'\leq \xi}A_{\xi'}$ and $B^c=A\setminus B=\Cup_{\xi<\xi'<\gamma}A_{\xi'}$. Since we assume $F(\alpha)^{|\alpha|}=F(\alpha)$, we see $\sup_{\beta\in B^c}h_\beta(\alpha)=\sup_{\xi<\xi'<\gamma}H_{\xi'}(\alpha)=H_{\xi+1}(\alpha)$ and therefore $H_{\xi+1}:\alpha\mapsto(\sup_{\beta\in B^c}h_\beta(\alpha))^{|\alpha|}$.

By \Cref{continuum size} we see that $(\Loc_{H_{\xi+1}})^{\b V[G\restriction B]}$ has cardinality $\kappa^+\cdot |B|=\kappa^+\cdot |\sup_{\xi'\leq \xi}A_{\xi'}|=\lambda_\xi$. Combining this with \Cref{having sacks product}, we see that there exists a family in $\b V[G]$ of size $\lambda_\xi$ that forms a witness for $\b V[G]\md\ap{\dstar{H_{\xi+1}}\leq \lambda_\xi}$.

Finally, for each $\eta<\kappa$ we see that if $\alpha\in S_{\eta}$ and $\beta\in A_{\sigma(\eta)}$, then $h_\beta(\alpha)=H_{\sigma(\eta)}(\alpha)=2^{F(\alpha)}$, while $g_\eta(\alpha)=F(\alpha)$, thus using \Cref{not having sacks product} we see that $\b V[G]\md\ap{\lambda_{\sigma(\eta)}=|A_{\sigma(\eta)}|\leq \dstar{g_\eta}}$.

In conclusion, $\b V[G]\md\ap{\lambda_{\sigma(\eta)}\leq\dstar{g_\eta}\leq\dstar{H_{\sigma(\eta)+1}}\leq\lambda_{\sigma(\eta)}}$ for each $\eta<\kappa$ 
\end{proof}
\begin{crl}
\renewcommand{\qedsymbol}{$\square$}
There exists functions $h_\xi$ for each $\xi<\kappa$ such that for any cardinals $\lambda_\xi>\kappa$ it is consistent that simultaneously $\dstar{h_\xi}=\lambda_\xi$ for all $\xi<\kappa$.\qedhere
\end{crl}

\nocite{*}
\footnotesize{
\bibliographystyle{alpha}
\bibliography{bibl}

\begin{thebibliography}{BBTFM18}

\bibitem[Bar87]{Bart87}
Tomek Bartoszy{\'n}ski.
\newblock Combinatorial aspects of measure and category.
\newblock {\em Fundamenta Mathematicae}, 127(3):225--239, 1987.

\bibitem[BBTFM18]{BBFM}
J{\"o}rg Brendle, Andrew Brooke-Taylor, Sy-David Friedman, and Diana~Carolina
  Montoya.
\newblock Cicho{\'n}’s diagram for uncountable cardinals.
\newblock {\em Israel Journal of Mathematics}, 225(2):959--1010, 2018.

\bibitem[BGS20]{BGS}
Thomas Baumhauer, Martin Goldstern, and Saharon Shelah.
\newblock The {H}igher {C}icho{\'n} {D}iagram.
\newblock {\em Fundamenta Mathematicae}, 252(3):241--314, 2020.

\bibitem[BJ95]{BJ}
Tomek Bartoszy{\'n}ski and Haim Judah.
\newblock {\em Set {T}heory: {O}n the {S}tructure of the {R}eal {L}ine}.
\newblock A.K. Peters, Wellesley, MA, 1995.

\bibitem[BL79]{BL79}
James~E Baumgartner and Richard Laver.
\newblock Iterated perfect-set forcing.
\newblock {\em Ann. Math. Logic}, 17(3):271--288, 1979.

\bibitem[Jec03]{Jech}
Thomas Jech.
\newblock {\em Set {T}heory: {T}hird {M}illennium {E}dition}.
\newblock Springer Monographs in Mathematics, 2003.

\bibitem[Kan80]{Kan80}
Akihiro Kanamori.
\newblock Perfect-set forcing for uncountable cardinals.
\newblock {\em Annals of Mathematical Logic}, 19(1-2):97--114, 1980.

\bibitem[Sac71]{Sacks71}
Gerald~E Sacks.
\newblock Forcing with perfect closed sets.
\newblock In {\em Axiomatic set theory}, volume~1, pages 331--355. Amer. Math.
  Soc Providence, RI, 1971.

\end{thebibliography}
}

\end{document}